\documentclass[12pt]{article}
\usepackage{amssymb,amsmath}
\usepackage{latexsym,bm}

\usepackage{amsfonts}
\usepackage{latexsym,amsmath,amssymb,amsfonts,epsfig,psfrag,url,graphics,ifpdf,multicol}
\usepackage{eepic,color,colordvi,amscd,amsthm}
\usepackage[section]{algorithm}
\usepackage{algpseudocode}
\usepackage[numbers,sort&compress]{natbib}
\usepackage{indentfirst,graphics,epsfig}
\usepackage{graphicx}
\usepackage{graphics}
\usepackage{tikz}

\newtheorem{theo}{Theorem}
\newtheorem{lem}[theo]{Lemma}

 \theoremstyle{definition}

\theoremstyle{remark}

\newcounter{casenum}[theo]

\newcounter{subcasenum}[theo]

\newcounter{claimnum}[theo]

\allowdisplaybreaks
\usepackage{indentfirst,subfig}

\begin{document}
\thispagestyle{plain}

\begin{center} {\Large Saturation Numbers for Linear Forests $P_6 + tP_2$
}
\end{center}
\pagestyle{plain}
\begin{center}
{
  {\small  Jingru Yan \footnote{ E-mail address: mathyjr@163.com}}\\[3mm]
  {\small  Department of Mathematics, East China Normal University, Shanghai 200241,  China }\\

}
\end{center}

\begin{center}

\begin{minipage}{140mm}
\begin{center}
{\bf Abstract}
\end{center}
{\small   A graph $G$ is $H$-saturated if it contains no $H$ as a subgraph, but does contain $H$ after the addition of any edge in the complement of $G$. The saturation number, $sat (n, H)$, is the minimum number of edges of a graph in the set of all $H$-saturated graphs with order $n$. In this paper, we determine the saturation number $sat (n, P_6 + tP_2)$ for $n \geq 10t/3 + 10$ and characterize the extremal graphs for $n >10t/3 + 20$.

 }

\end{minipage}
\end{center}

\section{Introduction}
In this paper we consider only simple graphs. For terminology and notations we follow the books \cite{BM,W}. Let $G$ be a graph with vertex set $V(G)$ and edge set $E(G)$. The order and the size of a graph $G$, denoted $|G|$ and $|E(G)|$, are its number of vertices and edges, respectively. For a vertex $v\in V(G)$, $d_G(v)$ is the degree of $v$ and $N_G(v)$ is the neighborhood of $v$. $N_G[v]=N_G(v)\cup \{v\}$. If the graph $G$ is clear from the context, we will omit it as the subscript. $\overline{G}$ and $\delta(G)$ denote the complement and minimum degree of a graph $G$, respectively. Denote by $G[A]$ the subgraph of $G$ induced by $A \subseteq V(G)$. $P_n, K_n$ and $S_n$ stand for \emph{path}, \emph{complete graph} and \emph{star} of order $n$, respectively.

Given graphs $G$ and $H$, a \emph{copy} of $H$ in $G$ is a subgraph of $G$ that is isomorphic to $H$. And the notation $G + H$ means the \emph{disjoint union} of $G$ and $H$. Then $tG$ denotes the disjoint union of $t$ copies of $G$. For graphs we will use equality up to isomorphism, so $G = H$ means that $G$ and $H$ are isomorphic.

A graph $G$ is \emph{$H$-saturated} if $G$ contains no $H$ as a subgraph but $G+e$ contains $H$ for any edge $e \in E(\overline{G})$. The set of $H$-saturated graphs of order $n$ is denoted by $SAT(n, H)$. $\overline{SAT}(n, H)$ and $\underline{SAT}(n, H)$ stand for the  set of $H$-saturated graphs with maximum number of edges and minimum number of edges, respectively. The number of edges in a graph in $\overline{SAT}(n, H)$ is Tur\'{a}n  number \cite{T}, denoted by $ex(n, H)$. The number of edges in a graph in $\underline{SAT}(n, H)$ is saturation number, denoted by $sat (n, H)$.

The first result about the saturation number of a graph was introduced by Erd\H{o}s, Hajnal, and Moon in \cite{EHM} in which the authors proved $sat(n,K_t) =\binom{t-2}{2}+(n-t+2)(t-2)$ and
$\underline{SAT}(n,K_t) = \{K_{t-2}\vee \overline{K}_{n-t+2}\}$, where $\vee $ denotes the \emph{join} of $K_{t-2}$ and $\overline{K}_{n-t+2}$, which is obtained from $K_{t-2} + \overline{K}_{n-t+2}$ by adding edges joining every vertex of $K_{t-2}$ to every vertex of $\overline{K}_{n-t+2}$. In addition to cliques, some of the graphs for which saturation number is known include unions of cliques \cite{BFO,FFGJ}, complete bipartite graphs \cite{B,C2,SW}, forests \cite{CFFGJM,FW}, books \cite{CFG}, small cycles \cite{C,ZT} and trees \cite{FFG,KT}.

In fact, both $sat(n, tP_2)$ and $\underline{SAT}(n, tP_2)$ are established by K\'{a}szonyi and Tuza in \cite{KT}. Chen et al. \cite{CFFGJM} focused on the saturation numbers for $P_k + tP_2$ with $k\geq 3$. Fan and Wang \cite{FW} determined the saturation number $sat(n, P_5 + tP_2)$ for $n \geq 3t + 8$ and characterized the extremal graphs for $n > (18t + 76)/5$, such as the following results.
\begin{theo}\label{th1}\cite{KT}
For $n\geq 3t-3$, $sat(n,tP_2)=3t-3$ and $\underline{SAT}(n,tP_2)= \{(t-1)K_3 + \overline{K}_{n-3t+3}\}$ or $t=2,n=4,\underline{SAT}(4,2P_2)=\{K_3 + K_1, S_4\}$.
\end{theo}

\begin{theo}\label{th2}\cite{CFFGJM}
For $n$ sufficiently large,\\
(1) $sat(n, P_3 + tP_2) = 3t$ and $tK_3 + \overline{K}_{n-3t}\in \underline{SAT}(n, P_3 + tP_2)$,\\
(2) $sat (n, P_4 + tP_2) = 3t+7$ and $K_5 + (t-1)K_3 + \overline{K}_{n-3t-2} \in \underline{SAT}(n, P_4 + tP_2)$.
\end{theo}

\begin{theo}\label{th3}\cite{FW}
Let $n$ and $t$ be two positive integers with $n\geq 3t+8$. Then,\\
(1) $sat(n, P_5 + tP_2) = min\{\lceil\frac{5n-4}{6}\rceil,3t+12\}$,\\
(2) $\underline{SAT}(n, P_5 + tP_2) = \{K_6 + (t-1)K_3 + \overline{K}_{n-3t-3}\}$ for $n > \frac{18t+76}{5}$ .
\end{theo}

In this paper, we further consider the saturation number of the linear forests $P_6 + tP_2$ with $t\geq 1$. The $t$ mentioned below all satisfy that $t\geq 1$.

\begin{theo}\label{th4}
Let $n$ and $t$ be two positive integers with $n\geq 10t/3+10$. Then,\\
(1) $sat(n, P_6 + tP_2) = min\{n-\lfloor\frac{n}{10}\rfloor,3t+18\}$,\\
(2) $\underline{SAT}(n,P_6 + tP_2) = \{K_7 + (t-1)K_3 + \overline{K}_{n-3t-4}\}$ for $n > \frac{10t}{3}+20$ .
\end{theo}

\section{Preliminaries}
For an integer $i\geq0$, let $V_i(G)=\{v\in V(G): d(v)=i\}$. In other words, $|V_0(G)|$ represents the number of isolated vertices in $G$. In this section, we list several lemmas and the result of the saturation numbers for linear forests $P_6 + tP_2$  with $|V_0(G)|\geq 2$.

\begin{lem}\label{lem1}
\rm{(Berge-Tutte Formula \cite{CB})} For a graph $G$,
$$\alpha'(G)=\frac{1}{2}min\{|G|+|S|-o(G-S):S\subseteq V(G)\},$$
where $\alpha'(G)$ is the matching number of $G$ and $o(G-S)$ is the number of odd components of $G-S$.
\end{lem}

\begin{lem}\label{lem2}\cite{CFFGJM}
Let $k_1,\ldots, k_m \geq 2$ be $m$ integers and $G$ be a $(P_{k_1} + P_{k_2}
+ \cdots + P_{k_m})$-saturated graph. If $d(x)=2$ and $N(x)=\{u,v\}$, then $uv\in E(G)$.
\end{lem}

\begin{lem}\label{lem3'}\cite{FW}
Let $G$ be a $(P_5 + tP_2)$-saturated graph. If $V_0(G)\neq\emptyset$, then $V_1(G) =\emptyset$. Moreover, for any $x \in V(G) \setminus V_0(G)$, we have
$$N_G[x]\cup \{w\}\subseteq V(H),$$
where $H$ is any copy of $P_5 + tP_2$ in $G + xw$ and $w$ is a vertex in $V_0(G)$.
\end{lem}
Using the same method as in Lemma \ref{lem3'}, we can get a more general result, which is the content in Lemma \ref{lem3}.
\begin{lem}\label{lem3}
Let $G$ be a $(P_k + tP_2)$-saturated graph with $k\geq 2$, $t\geq 1$. If $V_0(G)\neq\emptyset$, then $V_1(G) =\emptyset$. Moreover, for any $x \in V(G) \setminus V_0(G)$, we have
$$N_G[x]\cup \{w\}\subseteq V(H),$$
where $H$ is any copy of $P_k + tP_2$ in $G + xw$ and $w$ is a vertex in $V_0(G)$.
\end{lem}

A \emph{book} $B_k$ consists of $k$ triangles sharing one edge. A \emph{k-fan} $F_k$ consists of $k$ triangles sharing one vertex. $G$ is \emph{$H$-free} means $G$ does not contain $H$ as a subgraph.
\begin{lem}\label{lem4}
Let $G$ be a connected graph of order $n\geq 6$ and $\delta(G)\geq 2$. If $G$ satisfies\\
(1) $G$ is $P_6$-free and $G$ contains $P_4$ as a subgraph, and\\
(2) if $d(x)=2$ and $N(x)=\{u,v\}$, then $uv\in E(G)$,\\
then $G=B_i$, $i\geq 4$ or $G=F_j$, $j\geq 3$ with $n$ odd.
\end{lem}
\begin{proof}
Select a longest path $P$ in $G$, say $P=x_1,x_2,\ldots,x_k$. As $G$ satisfies condition (1), we have $4\leq k <6$. It is easily verified that there exists $x\notin V(P)$, and $N(x) \cap V(P)\neq  \emptyset$,  $N(x) \cap \{x_1,x_k\}= \emptyset$. We distinguish two cases.

{\it Case 1.} $k=4$.

Observe that if $|N(x) \cap \{x_2,x_3\}|=2$, then $G$ contains a path $x_1,x_2,x,x_3,x_4$, contradicting the fact that $P$ is a longest path. We conclude that, $|N(x)\cap \{x_2,x_3\}|=1$. Because of the symmetry of $x_2$ and $x_3$, suppose $x$ is adjacent to $x_2$. Since $\delta(G)\geq 2$, there is one vertex $y\in N(x)$ and $y\notin V(P)$. Thus $G$ contains a path $y,x,x_2,x_3,x_4$, contradicting $k=4$.

{\it Case 2.} $k=5$.

If $x$ is adjacent to $x_2$ or $x_4$, we assert that $N(x)\cap (V(G) \setminus V(P))=\emptyset$ and $x_3\notin N(x)$. Otherwise, $G$ contains a path with length at least 5, contradicting $k=5$. Since $\delta(G)\geq 2$, then $d(x)=2$ and $N(x)=\{x_2, x_4\}$. If $d(x_3)>2$, $y\in N(x_3) \setminus \{x_2,x_4\}$ (possibly $y=x_1$ or $y=x_5$), $G$ contains a path $y,x_3,x_2,x,x_4,x_5$ or $y,x_3,x_4,x,x_2,x_1$, contradicting the fact that $P$ is a longest path. Thus $d(x_3)=2$ and $N(x_3)=\{x_2, x_4\}$. As $G$ satisfies condition (2), $x_2$ is adjacent to $x_4$. Clearly, $N(x_1),N(x_5)\subseteq V(P)$. Since $\delta(G)\geq 2$, then $N(x_1)=\{x_2, x_4\}$ and $N(x_5)=\{x_2, x_4\}$. Hence $G[x_1,x_2,x_3,x_4,x_5,x]=B_4$. For any vertex $y\in V(G) \setminus (V(P)\cup \{x\})$, $y$ is adjacent to $x_2$ or $x_4$. Using the same method, we have $d(y)=2$ and $N(y)=\{x_2, x_4\}$. Hence $G=B_i$, $i\geq 4$.

If $x$ is adjacent to $x_3$, it is easy to check that $x$ is not adjacent to $x_2$ or $x_4$. Thus there is a vertex $y\in N(x)$ and $y\notin V(P)$. Note that $P$ is not a longest path if $N(y)\neq \{x,x_3\}$. If $x_1$ is adjacent to $x_4$, $G$ contains a path $x_4,x_1,x_2,x_3,x,y$, contradicting $k=5$. Thus $d(x_1)=2$ and $N(x_1)=\{x_2, x_3\}$. Similarly, $d(x_5)=2$ and $N(x_5)=\{x_3, x_4\}$. Now we consider the degrees of vertices $x,x_2$ and $x_4$. If any vertex of $\{x,x_2,x_4\}$ has degree more than two, $G$ has a path with length at least 5. Hence, $G[x_1,x_2,x_3,x_4,x_5,x,y]=F_3$. For any vertex $z\in V(G) \setminus (V(P)\cup \{x,y\})$, $z$ is adjacent to $x_3$. Using the same method, we have $G=F_i$, $i\geq 3$ with $n$ odd. This completes the proof of Lemma \ref{lem4}.
\end{proof}

\begin{theo}\label{lem5}
Let $G\in SAT(n, P_6  +  tP_2)$ and $Q=Q_1+Q_2+\cdots+Q_k$, where $Q_1,\ldots,Q_k$ are all the nontrivial components of $G$. If $|Q|\geq 2t+6, \delta(Q)\geq 2$, $|Q_i|\geq 6$ and $Q_i$ is not a book or fan, $1\leq i \leq k$, then\\
(1) $G \in SAT(n, P_4 + (t+1)P_2)$,\\
(2) if $V_0(G)\neq \emptyset$, then $|E(G)|>3t+18$.
\end{theo}
\begin{proof}
(1) Since $G\in SAT(n, P_6  +  tP_2)$, $G + e$ contains $P_6 + tP_2$ for any edge $e\in E(\overline{G})$. It follows that $G + e$ contains $P_4 + (t+1)P_2$ for any edge $e\in E(\overline{G})$.

If $G\notin SAT(n,P_4 + (t+1)P_2)$, then $G$ contains $P_4 + (t+1)P_2$. Without loss of generality, suppose that $Q_1$ contains $P_4$ as a subgraph. Since $|Q_1| \geq 6, \delta(Q)\geq 2$ and $Q_1$ is not a book or fan, by Lemma \ref{lem2} and Lemma \ref{lem4}, there exists $P_6$ in $Q_1$. Hence, $G$ contains a copy of $P_6 + tP_2$, a contradiction.

(2) Suppose that $|E(G)|\leq 3t+18$. By (1), we have $Q\in SAT(n, P_4 + (t+1)P_2)$. Then, $\alpha'(Q) \geq t + 2$. If $\alpha'(Q) \geq t+3$, $G$ must contain a copy of $(t+3)P_2$. Since $\delta(Q)\geq 2$ and $|Q_i|\geq 6 (1\leq i \leq k)$, it is clearly that $Q$ has a copy of $P_4 + (t + 1)P_2$, which contradicts $Q\in SAT(n, P_4 + (t+1)P_2)$. So, we have $\alpha'(Q)=t + 2$. By Lemma \ref{lem1}, we have
$$ t + 2 = \frac{1}{2} min\{|Q| + |X|- o(Q - X): X \subseteq V(Q)\}. $$
Choose a subset $Y \subseteq V(Q)$ such that
$$ t + 2 = \frac{1}{2}(|Q| + |Y| - o(Q - Y)).$$
Let $Q-Y=Q'_1+Q'_2+\cdots+Q'_p$. We have two claims.

{\it Claim 1.} $Q[Y\cup V(Q'_i)]$ is a complete graph for $i\in \{1,2,\ldots,p\}$.

To the contrary, suppose that there exist two vertices $u,v\in Y\cup V(Q'_i)$ such that $uv\notin E(Q)$. Let $Q'=Q+uv$. Since $Q$ is $(P_4 + (t+1)P_2)$-saturated, $\alpha'(Q')\geq t+3$. On the other hand, observe that $|Q'| = |Q|$ and $o(Q'-Y) = o(Q - Y)$. By Lemma \ref{lem1}, we have
$$ \alpha'(Q')\leq t + 2 = \frac{1}{2}(|Q'| + |Y| - o(Q' - Y)),$$ a contradiction.

{\it Claim 2.} $Y\neq \emptyset$.

Suppose that $Y = \emptyset$. By Claim 1, $Q'_1, \ldots, Q'_p$ are all complete graphs of order at least 6. Hence, $\delta(Q) \geq 5$ and
$$2|E(Q)| =\sum_{x\in V(Q)}d_Q(x)=\sum_{j=1}^{p}|Q'_j||Q'_j-1|\geq 5|Q|+|Q'_i||Q'_i-6|, 1 \leq i \leq p.$$
Since $|Q| \geq 2t +6$ and $|E(Q)| = |E(G)| \leq 3t +18$, we have $|Q| = 2t + 6$, $t = 1$ and $|Q'_i| = 6$ for $1 \leq i \leq p$. Thus, $8 = |Q| =6p$, a contradiction. This completes the proof of Claim 2.

Let $x\in Y$ and $w\in V_0(G)$. By Lemma \ref{lem3}, we have $N_Q[x]\cup \{w\} \subseteq V(H)$, where $H$ is a copy of $P_6 + tP_2$ in $G+xw$. Hence $|N_Q[x]\cup \{w\}| \leq |V(H)| = 2t+6$. On the other hand, By Claim 1, $|N_Q[x]\cup \{w\}|= |Q|+1\geq 2t+6+1=2t+7$, a contradiction. This completes the proof of Theorem \ref{lem5}.
\end{proof}

\begin{theo}\label{lem6}
Let $G \in SAT(n, P_6 + tP_2)$ with $n \geq 3t+6$. If $|V_0(G)| \geq 2$ and
$|E(G)| \leq 3t +18$, then $|E(G)| = 3t +18$ and $G=K_7 + (t-1)K_3 + \overline{K}_{n-3t-4}$.
\end{theo}
\begin{proof}
Since $|V_0(G)| \geq 2$, by Lemma \ref{lem3}, $V_1(G) = \emptyset$. Note that all the components of order 3, 4 or 5 in $G$ are complete. Let
$$G =G' + t_3K_3 + t_4K_4 + t_5K_5 + B + F,$$
where $t_k$ is the number of components of $G$ with order $k$, $k\in \{3,4,5\}$, $B$ is the graph consists of all the components $B_i$, $i\geq 4$, and $F$ is the graph consists of all the components $F_j$, $j\geq 3$. We denote $B_c$ and $F_c$ are the number of $B_i$, $i\geq 4$ and $F_j$, $j\geq 3$, respectively. Since $|B_i|\geq 6$, we have $|B|\geq 6B_c$.

Clearly $|V_0(G')| =|V_0(G)|\geq 2$. Note that joining two isolated vertices in $V_0(G')$ in $G$, we have a copy of $P_6 + tP_2$. Thus, $G'$ contains $P_6$. As $G \in SAT(n, P_6 + tP_2)$, we have $t_3 + 2t_4+2t_5+2B_c+(|F|-F_c)/2 \leq t - 1$. Let $t' = t-t_3-2t_4-2t_5-2B_c-(|F|-F_c)/2$. Then, $t' \geq 1$. Since $G \in SAT(n, P_6 + tP_2)$, we have $G' \in SAT(n', P_6 + t'P_2)$, where $n' = n-3t_3-4t_4-5t_5-|B|-|F|$.

Consider the graph $Q'$ obtained from $G'$ by deleting all trivial components. Clearly, every component of $Q'$ has order at least 6 and is not a book or fan. Note that $\delta(Q') \geq 2$ and $G'\in SAT(n,P_6 + t'P_2)$ with $V_0(G') \neq \emptyset$. Since $$|E(G')| =|E(G)|-3t_3-6t_4-10t_5-(2|B|-3B_c)-3((|F|-F_c)/2)$$
$$\leq 3t'+18-4t_5-(2|B|-9B_c)\leq 3t' +18,$$
by Theorem \ref{lem5}, we have $|Q'| \leq 2t' +5$. Note that joining two non-adjacent vertices in $Q'$, there is no copy of $P_6 + t'P_2$ in $G'$. Then $Q'$ is a complete graph. As $|V_0(G')|\neq \emptyset$, $|Q'| \geq 2t' +5$ and hence $Q'=K_{2t'+5}$. Moreover, $|E(Q')| = |E(G')| \leq 3t' + 18$. It follows that $t' = 1$ and $Q'=K_7$.

Since $G'=K_7 + (n'-7)K_1$ with $|E(G')| = 3t' + 18$, we have $t_5=0$ and $|B|=0$. Consequently
$$G =K_7 + (n'-7)K_1 + t_3K_3 + t_4K_4 + F.$$
Note that $G$ contains $P_6$. It is easy to verify that if $t_4 > 0$, joining the vertices in $K_4$ with the vertices in $K_7$ does not increase the number of $P_2$ in $G$. Similarly, if $|F|>0$, joining two non-adjacent vertices in $F_j$, $j\geq 3$ also does not increase the number of $P_2$ in $G$. Therefore, $t_4=0, |F|=0$ and $t_3 = t-1$. Hence $G=K_7 + (t-1)K_3 + \overline{K}_{n-3t-4}$. This completes the proof of Theorem \ref{lem6}.
\end{proof}

So far, we have proved that when $n \geq 3t+6$ and $|V_0(G)| \geq 2$, $sat(n, P_6 + tP_2)=3t +18$ and $\underline{SAT}(n, P_6 + tP_2)=\{K_7 + (t-1)K_3 + \overline{K}_{n-3t-4}\}.$
\section{Proof of Theorem \ref{th4}}
For a graph $H$, using the definition and notation in \cite{FW}, $SAT^{\ast}(n, H)$ and $sat^{\ast}(n, H)$ denote the set of $H$-saturated graphs $G$ of order $n$ with $|V_0(G)| = 0$ and the minimum number of edges in a graph in $SAT^{\ast}(n, H)$.

Let $T$ be the tree of order 10 as shown in Figure 1. Let $T^{\ast}$ be the tree of order $n=10+r$, $0\leq r\leq 9$, obtained from $S_{4+\lfloor\frac{r}{3}\rfloor}$ by attaching two leaves to each of the $2+\lfloor\frac{r}{3}\rfloor$ leaves of $S_{4+\lfloor\frac{r}{3}\rfloor}$ and attaching  $n-(4+\lfloor\frac{r}{3}\rfloor)-2(2+\lfloor\frac{r}{3}\rfloor)$ leaves to the remaining leaf of $S_{4+\lfloor\frac{r}{3}\rfloor}$.
\begin{center}
 \begin{tikzpicture}[>=stealth]
\draw(0,0)node[scale=1]{$\bullet$};
\draw(-1.2,-1)node[scale=1]{$\bullet$};
\draw(0,-1)node[scale=1]{$\bullet$};
\draw(1.2,-1)node[scale=1]{$\bullet$};
\draw[line width=0.6] (0,0)--(-1.2,-1);
\draw[line width=0.6] (0,0)--(0,-1);
\draw[line width=0.6] (0,0)--(1.2,-1);
\draw(0.3,-2)node[scale=1]{$\bullet$};
\draw(0.9,-2)node[scale=1]{$\bullet$};
\draw(1.5,-2)node[scale=1]{$\bullet$};
\draw(-0.3,-2)node[scale=1]{$\bullet$};
\draw(-0.9,-2)node[scale=1]{$\bullet$};
\draw(-1.5,-2)node[scale=1]{$\bullet$};
\draw[line width=0.6] (-1.2,-1)--(-0.9,-2);
\draw[line width=0.6] (-1.2,-1)--(-1.5,-2);
\draw[line width=0.6] (0,-1)--(-0.3,-2);
\draw[line width=0.6] (0,-1)--(0.3,-2);
\draw[line width=0.6] (1.2,-1)--(0.9,-2);
\draw[line width=0.6] (1.2,-1)--(1.5,-2);
\end{tikzpicture}

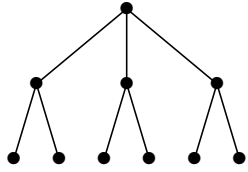
Figure 1. $T$
\end{center}

\begin{lem}\label{lem7}
Let $G$ be a $(P_6 + tP_2)$-saturated graph. If $T_1$ and $T_2$ are tree components of $G$, then $|T_1|\geq 10$, $|T_2|\geq 10$ and at least one of $T_1$ and $T_2$ contains $T$ as a subgraph.
\end{lem}
\begin{proof}
Let $v_i$ be a leaf of $T_i$ with $N(v_i) = \{u_i\}$, $i\in \{1, 2\}$. Since $G$ is $(P_6 + tP_2)$-saturated, $G + u_1u_2$ contains a copy of $P_6 + tP_2$. Let $H$ be the copy. If $u_1u_2$ is not in the $P_6$ of $H$, then $H-u_1u_2+u_1v_1$ is a copy of $P_6 + tP_2$ in $G$, contrary to $G$ is $(P_6 + tP_2)$-saturated. Thus $u_1u_2$ is in $P_6$ of $H$. It follows that $T_1 + T_2$ contains $P_4$ starting from $u_i$ for some $i = 1$ or 2 or  $T_1 + T_2$ contains $P_3$ starting from $u_i$ for $i = 1$ and $i = 2$. Now we discuss these two cases separately.

{\it Case 1.} $T_1 + T_2$ contains $P_4$ starting from $u_i$ for some $i = 1$ or 2.

Without loss of generality, assume $P_4= u_1,x,y,z$. Clearly $T_1[\{v_1,u_1,x,y,z\}]$ contains $P_5$. Let $M$ be the copy of $tP_2$ in $H$. Note that any vertex of $\{u_1,v_1,u_2,v_2,x,y,z\}$ is not in $M$. As $T_1$ is tree, by Lemma \ref{lem2}, $T_1$ has no vertex of degree 2. So, $u_1$, $x$ and $y$ all have neighbors not in $\{v_1,u_1,x,y,z\}$. Now we show that for any vertex $u'_1\in N(u_1) \setminus \{v_1,x\}$, $d(u'_1)=1$. If $d(u'_1)>1$ and $u'_1\in V(M)$. Then $u'_1$ has a neighbor $u''_1$ such that $u'_1u''_1$ belongs to $M$. Clearly, $T_1[\{u''_1,u'_1,u_1,x,y,z\}]$ contains $P_6$. Observe that $tP_2$ in $M- u'_1u''_1+ u_2v_2$. Hence $G$ contains $P_6 + tP_2$, a contradiction. If $d(u'_1)>1$ and $u'_1 \notin V(M)$, we also have $G$ contains $P_6 + tP_2$. Thus $d(u'_1)=1$. Using the same method, for any vertex $y'\in N(y) \setminus \{x,z\}$, we have $d(y')=1$. And the proof of $d(z)=1$ is similar to the above, so we omit it. Assume that $x$ has no neighbor $x'$ with $d(x')>1$, where $x'$ not equal to $u_1$ or $y$. The additional edge $e=u_1y$ in $G$ does not increase the number of $P_2$ and $T_1$ does not contain $P_6$, contradicting $G \in SAT(n, P_6 + tP_2)$. Hence $x$ has at least one neighbor of degree more than 1. So, $T_1$ contains $T$.

Next we show that for any vertex $x'\in N(x)$ with $d(x')>1$, $N(x')\setminus \{x\}$ are leaves. We distinguish two cases.

{\it Subcase 1.} $x'\notin V(M)$.
If there exists $x''\in N(x')$ with $d(x'')>1$, we have two cases. One is $x''\in V(M)$. Let $x'''$ is the neighbor of $x''$ such that $x''x'''$ belongs to $M$. Then we have $T_1[\{x''',x'',x',x,y,z\}]$ contains $P_6$ and uses one edge in $M$. By replacing $x''x'''$ with $u_1v_1$, we get a copy of $P_6 + tP_2$ in $G$. Another is $x''\notin V(M)$. Whether $x'''$ belongs to $V(M)$ or not, using the same method, we all have $G$ contains $P_6 + tP_2$, a contradiction.

{\it Subcase 2.} $x'\in V(M)$.
If there exists $x''\in N(x')$ with $d(x'')>1$, we can use the same method of Subcase 1 to check $T_1$ contains a copy of $P_6$ by using at most two edges of $M$. By replacing these two edges with $u_1v_1$ (or $yz$) and $u_2v_2$, we get a copy of $P_6 + tP_2$ in $G$, contrary to $G$ is a $(P_6 + tP_2)$-saturated graph.

Recall that $v_2$ be a vertex of $T_2$ with $N(v_2) = \{u_2\}$. Since $G$ is $(P_6 + tP_2)$-saturated, there is $P_6 + tP_2$ in $G + xu_2$ containing the edge $xu_2$. Let $H'$ be the copy and $M'$ be the copy of $tP_2$ in $H'$. If $xu_2$ is not in the $P_6$, by replacing $xu_2$ with $u_2v_2$, we have $P_6 + tP_2$ in $G$, a contradiction. Thus $xu_2$ is in the copy of $P_6$. Since $T_1$ does not contain a path of length 3 with $x$ as its endpoint, $T_2$ contains a path $P'$ of length 2 with $u_2$ as its endpoint. Hence $T_2[V(P')\cup \{v_2\}]$ contains a path $P$ of length 3, $P=v_2,u_2,w_1,w_2$.

Now we show that $T_2$ contains $T$ or $|T_2|\geq 10$. If $d(w_2)\neq 1$, it is easy to prove that there is one vertex in $N(w_2)\setminus \{w_1\}$ is not in $M'$. Hence $T_2$ contains $P_4$ starting from $u_2$. Using the same proof of $T_1$ contains $P_4$ starting from $u_1$, we have $T_2$ contains $T$ as a subgraph. If $d(w_2)= 1$ and $N(w_2)=\{w_1\}$. As $T_2$ is tree, by Lemma \ref{lem2}, $T_2$ has no vertex of degree 2. So, $u_2$ and $w_1$ all have neighbors not in $V(P)$. Note that for any vertex $u'_2\in N(u_2)\setminus V(P)$ or $w'_1\in N(w_1)\setminus V(P)$, if
there is one vertex of $(N(u'_2)\setminus \{u_2\})\cup (N(w'_1)\setminus \{w'_1\})$ is non-leaf, then $T_2$ contains $P_6$. Hence, by Lemma \ref{lem2}, $|T_2|\geq 10$. On the other hand, any vertex of $(N(u'_2)\setminus \{u_2\})\cup (N(w'_1)\setminus \{w'_1\})$ has degree at most 1. Assert that there are two non-leaves adjacent to the same vertex of $\{u_2,w_1\}$, then we have $T_2$ contains $T$ and complete the proof of Case 1. Otherwise, we have two cases. One is at most one vertex $w$ of $N(u_2)\cup N(w_1)$ with $d(w)\neq 1$, joining $w$ with $u_2$ or $w_1$ in $G$ does not increase the number of $P_2$ and $P_6$, contradicting $G \in SAT(n, P_6 + tP_2)$. Another is exactly there is one non-leaf, denoted $u'_2$, adjacent to $u_2$ and one non-leaf, denoted $w'_1$, adjacent to $w_1$. Considering the condition of any vertex of $(N(u'_2)\setminus \{u_2\}) \cup (N(w'_1)\setminus \{w'_1\})$ has degree at most 1, it is easy to check that $T_2$ contains $P_6$ and adding an edge $u'_2w'_1$ to $G$ will not increase the number of $P_2$, contrary to $G$ is $(P_6 + tP_2)$-saturated.

{\it Case 2.} $T_1 + T_2$ contains $P_3$ starting from $u_i$ for $i = 1$ and $i = 2$.

Denote by $P_3=u_1,x,y$ in $T_1$ and $P_3=u_2,w_1,w_2$ in $T_2$. Next, we only prove that $T_1$ contains $T$, and $T_2$ contains $T$ is similar. Clearly $T_1[\{v_1,u_1,x,y\}]$ contains $P_4$. Let $M''$ be the copy of $tP_2$ in $H$. Note that any vertex of $\{u_1,v_1,u_2,v_2,x,y,w_1,w_2\}$ is not in $M''$. Then $T_2$ contains two copies of $P_2$ not in $M''$. For both cases $d(y)\neq 1$ and $d(y)= 1$, we can use a proof similar to Claim 1 to prove. So we omit it. This completes the proof of Lemma \ref{lem7}.
\end{proof}

\begin{theo}\label{th5}
For $n\geq 10t/3+10$, $sat^{\ast}(n, P_6 + tP_2)=n-\lfloor\frac{n}{10}\rfloor$.
\end{theo}
\begin{proof}
Suppose $sat^{\ast}(n, P_6 + tP_2)< n-\lfloor\frac{n}{10}\rfloor$, then there is a graph $G\in SAT^{\ast}(n, P_6 + tP_2)$ with $|E(G)|<n-\lfloor\frac{n}{10}\rfloor$. Let $G = R + (T_1 + \cdots + T_k)$, where $T_1, \ldots, T_k$ are all the
tree components of $G$. Hence,
$$|E(G)| = |E(R)| +\sum_{i=1}^{k}|E(T_{i})|\geq |R|+\sum_{i=1}^{k}(|T_{i}|-1)=|G|-k=n-k.$$
Since $|E(G)|< n-\lfloor\frac{n}{10}\rfloor$, we have $k>\lfloor\frac{n}{10}\rfloor$. As $k\geq 2$, by Lemma \ref{lem7}, $|T_i|\geq 10$ for $1 \leq i \leq k$. Hence, $n\geq 10k$, contrary to $k>\lfloor\frac{n}{10}\rfloor$. It follows that $sat^{\ast}(n, P_6 + tP_2)\geq n-\lfloor\frac{n}{10}\rfloor$.

On the other hand, denote $n = 10q +r$, where $q=\lfloor\frac{n}{10}\rfloor$, $0 \leq r \leq 9$. Since $n \geq 10t/3+10$, we have $10q +r\geq 10t/3+10$. Then
$$t\leq 3q+\lfloor\frac{3r}{10}\rfloor-3\leq 3q+\lfloor\frac{r}{3}\rfloor-3.$$
Consider the graph
$$G^{\ast}=(q-1)T + T^{\ast}.$$
Obviously $G^{\ast}$ contains no copy of $P_6$ and $G^{\ast}+e$ contains a copy of $P_6 + (3q+\lfloor\frac{r}{3}\rfloor-3)P_2$ for any $e\in E(\overline{G^{\ast}})$. This implies that $G^{\ast}$ is $(P_6 + tP_2)$-saturated. Since $|V_0(G^{\ast})|=0$, $G^{\ast} \in SAT^{\ast}(n, P_6 + tP_2)$. Hence $sat^{\ast}(n, P_6 + tP_2)=E(G^{\ast})=n-\lfloor\frac{n}{10}\rfloor$. This completes the proof of Theorem \ref{th5}.
\end{proof}
Finally, we show that the proof of Theorem \ref{th4}.
\begin{proof}
(1) Suppose $G$ is $(P_6+tP_2)$-saturated. If $|V_0(G)| = 1$, by Lemma \ref{lem3}, $V_1(G) = \emptyset$. By degree-sum formula,
$$2|E(G)| =\sum_{x\in V(G)}d(x)\geq 2(|G|-1).$$ For $n\geq \frac{10t}{3}+10$, $|E(G)|\geq |G|-1=n-1>n-\lfloor\frac{n}{10}\rfloor\geq min\{n-\lfloor\frac{n}{10}\rfloor,3t+18\}$. If $|V_0(G)| = 0$ or $|V_0(G)| \geq 2$, by Theorem \ref{lem6} and Theorem \ref{th5}, we have $sat(n,P_6 + tP_2)= min\{n-\lfloor\frac{n}{10}\rfloor,3t+18\}$ for $n\geq \frac{10t}{3}+10$. This complete the proof.

(2) By $n>\frac{10t}{3}+20$, we have $n-\lfloor\frac{n}{10}\rfloor >3t+18$. Consequently $sat(n,P_6 + tP_2)=3t+18$. Let $G\in SAT(n, P_6 + tP_2)$ with $|E(G)| = 3t + 18$. By Theorem \ref{th5}, we have $G \notin SAT^{*}(n,P_6 + tP_2)$ and hence $|V_0(G)|\neq 0$. If $|V_0(G)| = 1$, we obtain that
$$|E(G)|\geq |G|- 1>\frac{10t}{3}+20-1=\frac{10t}{3}+19>3t+18,$$ a contradiction. Thus $|V_0(G)| \geq 2$. By Theorem \ref{lem6}, we have $\underline{SAT}(n, P_6 + tP_2) = \{K_7 + (t-1)K_3 + \overline{K}_{n-3t-4}\}$. This completes the proof of Theorem \ref{th4}.
\end{proof}

{\bf Acknowledgement}
This research was supported Science and Technology Commission of Shanghai Municipality (STCSM) grant 18dz2271000.

\end{document}